\documentclass[11pt]{article}
\usepackage[a4paper,margin=28mm]{geometry}
\usepackage{amsmath,amssymb,amsthm,mathtools}
\usepackage{microtype}
\usepackage[hidelinks]{hyperref}
\usepackage{enumitem}
\usepackage{tikz}
\newtheorem{theorem}{Theorem}[section]
\newtheorem{proposition}[theorem]{Proposition}
\newtheorem{lemma}[theorem]{Lemma}
\newtheorem{corollary}[theorem]{Corollary}
\newtheorem{conjecture}[theorem]{Conjecture}
\theoremstyle{remark}
\newtheorem{remark}[theorem]{Remark}

\newcommand{\width}{\operatorname{width}}
\newcommand{\Var}{\operatorname{Var}}
\newcommand{\Cov}{\operatorname{Cov}}
\newcommand{\N}{\mathbb N}

\title{Antichains among Divisor Sums of Divisors}
\author{Felix Huber\\
\small Independent researcher, Lucerne, Switzerland\\
\small \texttt{felix.68@gmx.ch}\\
\small ORCID: \href{https://orcid.org/0009-0005-1568-1579}{0009-0005-1568-1579}}
\date{}

\begin{document}
\maketitle

\begin{abstract}
For a positive integer \(n\), let
\[
S(n)=\{\sigma(d):d\mid n\}
\]
be ordered by divisibility, and let \(a(n)\) be its width. We study how much of the divisor lattice of \(n\) survives under the map \(d\mapsto\sigma(d)\). For every fixed prime-exponent pattern, the largest possible value of \(a(n)\) is the width of a corresponding product of divisibility posets, and this maximum is attained infinitely often. In the squarefree case with \(\omega(n)=r\), every width from \(1\) to \(\binom r{\lfloor r/2\rfloor}\) occurs, again infinitely often.

For repeated exponent patterns, Engel's limit theory reduces the asymptotic width problem to optimal representation variance. We show that for the divisibility poset \(D_m=(\{1,\ldots,m\},\mid)\), the prime directions \(2\) and \(3\) already determine the exact variance. The structural reason is the critical dilation \(4\), lying strictly between \(3\) and the next prime \(5\). This yields an exact finite formula, the limit \(V(D_m)\to 1.0367934201598\ldots\), and asymptotics for repeated exponent patterns. The complementary minimum-width problem remains largely open.
\end{abstract}



\section{Introduction}
For a positive integer \(n\), consider
\[
S(n)=\{\sigma(d):d\mid n\}.
\]
We order this finite set by divisibility and write
\[
a(n)=\width(S(n),\mid).
\]
Thus \(a(n)\) is the largest number of distinct divisor sums, taken from divisors of the same integer \(n\), such that no one divides another.

For example, if \(n=6\),
\[
S(6)=\{1,3,4,12\}.
\]
The two middle values are incomparable, so \(a(6)=2\). The initial values of the sequence are
\[
1,1,1,2,1,2,1,2,2,1,1,3,1,2,2,3,1,3,1,2,\ldots.
\]

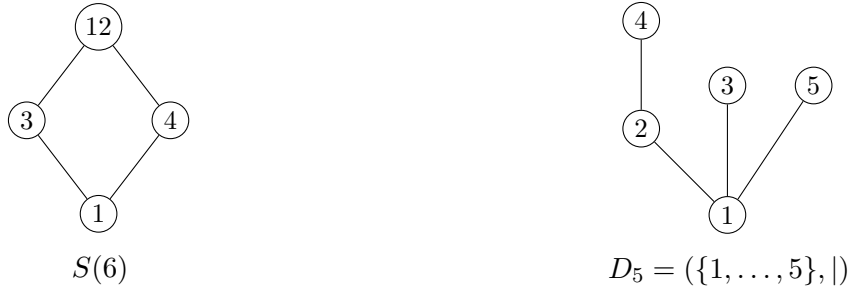
\begin{figure}[ht]
\centering
\begin{minipage}{0.46\textwidth}
\centering
\begin{tikzpicture}[scale=0.95, every node/.style={circle,draw,inner sep=2pt,font=\small}]
  \node (one) at (0,0) {$1$};
  \node (three) at (-1,1.3) {$3$};
  \node (four) at (1,1.3) {$4$};
  \node (twelve) at (0,2.6) {$12$};
  \draw (one)--(three)--(twelve);
  \draw (one)--(four)--(twelve);
\end{tikzpicture}

\smallskip
$S(6)$
\end{minipage}
\hfill
\begin{minipage}{0.46\textwidth}
\centering
\begin{tikzpicture}[scale=0.95, every node/.style={circle,draw,inner sep=2pt,font=\small}]
  \node (d1) at (0,0) {$1$};
  \node (d2) at (-1.2,1.2) {$2$};
  \node (d3) at (0,1.8) {$3$};
  \node (d5) at (1.2,1.8) {$5$};
  \node (d4) at (-1.2,2.7) {$4$};
  \draw (d1)--(d2)--(d4);
  \draw (d1)--(d3);
  \draw (d1)--(d5);
\end{tikzpicture}

\smallskip
$D_5=(\{1,\ldots,5\},\mid)$
\end{minipage}
\caption{Two basic divisibility posets. The first nontrivial set $S(6)=\{1,3,4,12\}$ is a diamond. For a prime fourth power, the divisor sums have the order structure of $D_5$; this is proved in Theorem~\ref{thm:primepower}.}
\label{fig:first-posets}
\end{figure}

Sets with no internal divisibility are usually called \emph{primitive sets}; they form a classical meeting point of multiplicative number theory and extremal combinatorics \cite{ErdosPrimitive,Lichtman2023}. Here the map
\[
d\longmapsto \sigma(d)
\]
sends the divisor lattice of one integer into the divisibility poset of the positive integers, and we ask how much order structure survives. For a prime power, divisibility among the values \(\sigma(p^{k-1})=1+p+\cdots+p^{k-1}\) is controlled exactly by divisibility of the index \(k\). With several prime factors, additional \emph{cross-divisibilities} between different prime bases become the main arithmetic difficulty.

For a fixed exponent pattern \(\mathbf e=(e_1,\ldots,e_r)\), we show that the primes can be chosen so that only the index-forced relations remain. Hence the maximal problem becomes the width problem for
\[
D_{e_1+1}\times\cdots\times D_{e_r+1}.
\]
In the squarefree case the extremal widths are \(1\) and the central binomial coefficient; a one-parameter threshold construction shows that every intermediate width occurs as well.

Repeated exponent patterns lead to powers of \(D_m=(\{1,\ldots,m\},\mid)\), bringing in Engel's theory of optimal representations and limit Sperner theorems \cite{Engel1986,EngelBook}. Divisibility orders on initial intervals have also been studied through order dimension \cite{LewisSouza,Haiman2024}; here the relevant invariant is the minimum representation variance \(V(D_m)\). We prove that the prime directions \(2\) and \(3\) already determine this variance. The critical dilation is \(4\), and
\[
3<4<5
\]
explains why \(3\) must be retained while every prime at least \(5\) is redundant. This yields an exact formula for \(V(D_m)\), its limiting constant, and asymptotics for repeated exponent patterns.

Thus cross-divisibilities are eliminated in the maximal problem, controlled in the squarefree spectrum, and reduced to their active prime directions in the variance problem. The complementary minimum-width problem is discussed in Section~7 and remains largely open.

\clearpage
\section{Basic properties and prime powers}

\begin{proposition}\label{prop:monotone}
If $m\mid n$, then $a(m)\le a(n)$.
\end{proposition}
\begin{proof}
Every divisor of $m$ is a divisor of $n$, hence $S(m)\subseteq S(n)$. Every antichain in $S(m)$ is therefore an antichain in $S(n)$.
\end{proof}

For a prime $p$, put
\[
R_k(p)=1+p+\cdots+p^{k-1}=\frac{p^k-1}{p-1}.
\]

\begin{lemma}\label{lem:repunit-div}
For $i,j\ge1$, one has $R_i(p)\mid R_j(p)$ if and only if $i\mid j$.
\end{lemma}
\begin{proof}
If $i\mid j$, the divisibility follows from the factorization of $x^j-1$. Conversely,
\[
\gcd(R_i(p),R_j(p))=R_{\gcd(i,j)}(p).
\]
Thus $R_i(p)\mid R_j(p)$ forces $i=\gcd(i,j)$, hence $i\mid j$.
\end{proof}

\begin{theorem}\label{thm:primepower}
For every prime $p$ and every $e\ge0$,
\[
\boxed{
a(p^e)=\left\lceil\frac{e+1}{2}\right\rceil.
}
\]
\end{theorem}
\begin{proof}
Since
\[
S(p^e)=\{R_1(p),\ldots,R_{e+1}(p)\},
\]
Lemma~\ref{lem:repunit-div} identifies this poset with $D_{e+1}$. The integers
\[
\left\lfloor\frac{e+1}{2}\right\rfloor+1,\ldots,e+1
\]
form an antichain. Conversely, partition $1,\ldots,e+1$ according to odd part; every class is a chain of the form $u,2u,4u,\ldots$. There are $\lceil(e+1)/2\rceil$ such chains, so Dilworth's theorem \cite{Dilworth} gives the result.
\end{proof}

\begin{corollary}
Every positive integer occurs as a value of $a(n)$.
\end{corollary}
\begin{proof}
For $k\ge1$, Theorem~\ref{thm:primepower} gives $a(p^{2k-2})=k$.
\end{proof}

\begin{proposition}\label{prop:a1-squarefree}
If $a(n)=1$, then $n$ is squarefree.
\end{proposition}
\begin{proof}
If $p^2\mid n$, then $p+1$ and $p^2+p+1$ belong to $S(n)$. The latter is larger than $p+1$ and is congruent to $1$ modulo $p+1$, so these two values are incomparable.
\end{proof}

\begin{proposition}\label{prop:pq}
Let $p<q$ be primes. Then
\[
a(pq)=
\begin{cases}
1,&p+1\mid q+1,\\
2,&p+1\nmid q+1.
\end{cases}
\]
\end{proposition}
\begin{proof}
We have
\[
S(pq)=\{1,p+1,q+1,(p+1)(q+1)\}.
\]
The only possible incomparability is between $p+1$ and $q+1$.
\end{proof}

\section{Squarefree integers}
Suppose
\[
n=p_1\cdots p_r
\]
is squarefree. For $I\subseteq[r]$,
\[
\sigma\!\left(\prod_{i\in I}p_i\right)
=\prod_{i\in I}(p_i+1).
\]

\begin{theorem}\label{thm:sq-upper}
If $n$ is squarefree and $\omega(n)=r$, then
\[
\boxed{
a(n)\le \binom{r}{\lfloor r/2\rfloor}.
}
\]
\end{theorem}
\begin{proof}
Choose one subset $I\subseteq[r]$ representing each value in an antichain of $S(n)$. If $I\subseteq J$, then the product indexed by $I$ divides the product indexed by $J$, which is impossible for two distinct values in the chosen antichain. The chosen subsets therefore form an antichain in the Boolean lattice $B_r$. Sperner's theorem \cite{Sperner} gives the bound.
\end{proof}

\begin{theorem}\label{thm:sq-min}
For every $r\ge1$, infinitely many squarefree integers $n$ satisfy
\[
\omega(n)=r,
\qquad
a(n)=1.
\]
\end{theorem}
\begin{proof}
The assertion is clear for $r=1$. Proceed by induction. Suppose $n$ is squarefree and $S(n)$ is a chain, and put $M=\sigma(n)$. By Dirichlet's theorem \cite{Dirichlet}, there are infinitely many primes
\[
q\equiv-1\pmod M.
\]
Choose one with $q\nmid n$.
Then
\[
S(nq)=S(n)\cup(q+1)S(n).
\]
The two copies are chains, and every element of the first copy divides every element of the second because every element of $S(n)$ divides $M$ and hence divides $q+1$. Thus the union is again a chain.
\end{proof}

\section{Fixed exponent patterns}
The prime-power case suggests a natural model for a general exponent vector. If
\[
n=p_1^{e_1}\cdots p_r^{e_r},
\]
then every divisor sum is a product of repunits \(R_k(p_i)\). If one ignores divisibilities between different prime bases, the index set is the product poset
\[
P_{\mathbf e}
=
D_{e_1+1}\times\cdots\times D_{e_r+1},
\qquad
\mathbf e=(e_1,\ldots,e_r).
\]
The next theorem shows that this naive model is the exact maximal one: all additional cross-divisibilities can be avoided simultaneously.

\begin{theorem}\label{thm:fixed-pattern}
For every fixed exponent vector $\mathbf e$,
\[
\boxed{
\max_{p_1,\ldots,p_r}
 a(p_1^{e_1}\cdots p_r^{e_r})
=\width(P_{\mathbf e}),
}
\]
where the maximum is over distinct primes. Moreover, the maximum is attained for infinitely many choices of $p_1,\ldots,p_r$.
\end{theorem}
\begin{proof}
For
\[
\mathbf k=(k_1,\ldots,k_r)\in P_{\mathbf e}
\]
define
\[
F(\mathbf k)=\prod_{i=1}^r R_{k_i}(p_i).
\]
If $k_i\mid\ell_i$ for every $i$, then $F(\mathbf k)\mid F(\boldsymbol\ell)$. Given an antichain among the distinct values $F(\mathbf k)$, choose one index vector representing each value. No two chosen vectors can be comparable in $P_{\mathbf e}$, so they form an antichain. Therefore
\[
a(p_1^{e_1}\cdots p_r^{e_r})\le\width(P_{\mathbf e}).
\]

It remains to realize exactly the coordinatewise divisibility relations. Put
\[
m=\max_i(e_i+1).
\]
For every pair $(i,k)$ with $2\le k\le e_i+1$, choose a distinct prime $q_{i,k}>m$ satisfying
\[
q_{i,k}\equiv1\pmod k,
\]
and choose $g_{i,k}$ of exact multiplicative order $k$ modulo $q_{i,k}$. For each fixed $i$, prescribe
\[
p_i\equiv g_{i,k}\pmod{q_{i,k}}
\]
for the primes belonging to the $i$-th coordinate, and
\[
p_i\equiv1\pmod{q_{j,k}}
\qquad(j\ne i)
\]
for all primes belonging to the other coordinates. The Chinese remainder theorem gives a reduced residue class, and Dirichlet's theorem supplies infinitely many primes in it. The $p_i$ may therefore be chosen distinct.

For $2\le k\le e_i+1$ and $1\le u\le e_i+1$,
\[
q_{i,k}\mid R_u(p_i)
\quad\Longleftrightarrow\quad
k\mid u,
\]
because $p_i$ has exact order $k$ modulo $q_{i,k}$. On the other hand, if $j\ne i$, then
\[
R_u(p_j)\equiv u\pmod{q_{i,k}},
\]
and $1\le u\le m<q_{i,k}$, so $q_{i,k}$ divides none of the relevant factors from another coordinate.

Thus $q_{i,k}$ is a private witness for divisibility by $k$ in the $i$-th coordinate. If some $k_i$ does not divide $\ell_i$, then $q_{i,k_i}$ divides $F(\mathbf k)$ but not $F(\boldsymbol\ell)$. Consequently
\[
F(\mathbf k)\mid F(\boldsymbol\ell)
\quad\Longleftrightarrow\quad
k_i\mid\ell_i\ \text{for all }i.
\]
Hence the divisibility poset is isomorphic to $P_{\mathbf e}$, proving equality. Since each prescribed residue class contains infinitely many primes, infinitely many choices attain the maximum.
\end{proof}

\begin{corollary}\label{cor:sq-max}
For every $r\ge1$,
\[
\boxed{
\max_{\substack{n\text{ squarefree}\\\omega(n)=r}}a(n)
=\binom r{\lfloor r/2\rfloor}.
}
\]
The maximum is attained infinitely often.
\end{corollary}
\begin{proof}
For the squarefree pattern, $P_{\mathbf e}=D_2^r$, which is the Boolean lattice $B_r$. By Sperner's theorem,
\[
\width(B_r)=\binom r{\lfloor r/2\rfloor},
\]
and Theorem~\ref{thm:fixed-pattern} gives the result, including infinitely many attaining choices.
\end{proof}

Combining Theorem~\ref{thm:sq-min} and Corollary~\ref{cor:sq-max} shows that both squarefree extremal values are sharp and occur infinitely often.

\section{The complete squarefree spectrum}\label{sec:spectrum}
For
\[
\mathcal W_r=\{a(n):n\text{ squarefree},\ \omega(n)=r\},
\qquad
b_r=\binom r{\lfloor r/2\rfloor},
\]
we prove \(\mathcal W_r=\{1,\ldots,b_r\}\). The two ingredients are a one-prime extension that preserves width and a weighted construction for the new values between \(b_r\) and \(b_{r+1}\).

\begin{lemma}[Controlled one-prime extension]\label{lem:controlled-extension}
Let $n$ be squarefree, put $N=\sigma(n)$, and let $M\mid N$. Assume that $2\mid M$ whenever $2\mid N$. Then there are infinitely many primes $q>N$ for which
\[
q+1=Mk,
\qquad
\gcd(k,N)=1.
\]
For every such $q$ and all $x,z\in S(n)$,
\[
\boxed{
 x\mid(q+1)z
 \quad\Longleftrightarrow\quad
 \frac{x}{\gcd(x,z)}\mid M.
}
\tag{5.1}
\]
\end{lemma}
\begin{proof}
We first choose a residue class $c$ modulo $N$. For every prime $\ell\mid N$ with $\ell\mid M$, require $c\not\equiv0\pmod\ell$. For every prime $\ell\mid N$ with $\ell\nmid M$, the prime $\ell$ is odd by the hypothesis on the factor $2$; choose $c$ modulo $\ell$ avoiding both $0$ and $M^{-1}$. These local conditions can be lifted to the prime-power factors of $N$, and the Chinese remainder theorem gives $c$ such that
\[
\gcd(c,N)=1,
\qquad
\gcd(Mc-1,N)=1.
\]
Since also $Mc-1\equiv-1\pmod M$, we have
\[
\gcd(Mc-1,MN)=1.
\]
Dirichlet's theorem therefore gives infinitely many primes
\[
q\equiv Mc-1\pmod{MN}.
\]
Choosing one with $q>N$, write
\[
q+1=M(c+Nt)=Mk.
\]
Then $\gcd(k,N)=1$.

Because $n$ is squarefree, every element of $S(n)$ divides $N$. Put
\[
g=\gcd(x,z),\qquad x=gx',\qquad z=gz',
\]
so that $\gcd(x',z')=1$. Then
\[
x\mid(q+1)z
\quad\Longleftrightarrow\quad
x'\mid Mkz'.
\]
Now $x'\mid N$, hence $\gcd(x',k)=1$, and $\gcd(x',z')=1$. Thus the last divisibility is equivalent to $x'\mid M$, which is exactly (5.1).
\end{proof}

\begin{corollary}\label{cor:spectrum-preserve}
For every $r\ge1$,
\[
\mathcal W_r\subseteq\mathcal W_{r+1}.
\]
More precisely, if a width occurs for a squarefree integer with $r$ prime factors, then it occurs for infinitely many squarefree integers with $r+1$ prime factors.
\end{corollary}
\begin{proof}
Apply Lemma~\ref{lem:controlled-extension} with $M=N$. Then every $x,z\in S(n)$ satisfy
\[
x\mid(q+1)z,
\]
so
\[
S(nq)=S(n)\oplus(q+1)S(n)
\]
is the ordinal sum of two copies of $S(n)$. Its width is therefore $a(n)$. The lemma gives infinitely many choices of $q$.
\end{proof}

To realize the missing widths between $b_r$ and $b_{r+1}$, start from an exact Boolean lattice and add one prime whose cross-relations between the two Boolean copies are controlled by a single integer threshold.

Fix $r\ge1$ and put
\[
A=2^r,
\qquad
w_i=A+2^{i-1}
\quad(1\le i\le r).
\]
Choose distinct auxiliary primes $\ell_1,\ldots,\ell_r>3$. By the Chinese remainder theorem and Dirichlet's theorem, we may choose distinct primes $p_1,\ldots,p_r>3$ such that
\[
p_i\equiv 3^{w_i}-1\pmod{3^{w_i+1}},
\tag{5.2}
\]
\[
p_i\equiv-1\pmod{\ell_i},
\qquad
p_i\equiv1\pmod{\ell_j}\quad(j\ne i).
\tag{5.3}
\]
The residue classes are reduced, so infinitely many choices are available. Put
\[
n=p_1\cdots p_r,
\qquad
x_I=\prod_{i\in I}(p_i+1)
\quad(I\subseteq[r]).
\]
Equation (5.3) makes $\ell_i$ a private prime divisor of $p_i+1$. Consequently
\[
x_I\mid x_J
\quad\Longleftrightarrow\quad
I\subseteq J,
\tag{5.4}
\]
so $S(n)$ is exactly a Boolean lattice $B_r$. By (5.2),
\[
v_3(p_i+1)=w_i.
\]
Define
\[
W(I)=v_3(x_I)
=A|I|+\sum_{i\in I}2^{i-1}.
\tag{5.5}
\]
These weights are pairwise distinct integers. Indeed, for fixed $|I|$ the binary sum
$\sum_{i\in I}2^{i-1}$ determines $I$, while different cardinalities lie in disjoint
intervals $[kA,(k+1)A-1]$.
Let
\[
H=W([r])=v_3(\sigma(n)).
\]
For $0\le t\le H$, put
\[
M_t=\frac{\sigma(n)}{3^{H-t}}.
\]
Since the primes $p_i$ are odd, $M_t$ is even. Lemma~\ref{lem:controlled-extension} therefore gives infinitely many primes $q_t>\sigma(n)$ satisfying its conclusion with $M=M_t$.

Let $P_t=S(nq_t)$ as a divisibility poset. It consists of lower and upper copies $x_I$ and $(q_t+1)x_J$ of $B_r$; since $q_t>\sigma(n)$ there are no divisibilities from the upper copy to the lower. By (5.1), only the $3$-adic part can obstruct a lower-to-upper divisibility, and
\[
v_3\!\left(\frac{x_I}{\gcd(x_I,x_J)}\right)
=\max\{W(I)-W(J),0\}.
\]
Hence the cross-relations are exactly
\[
\boxed{
x_I\mid(q_t+1)x_J
\quad\Longleftrightarrow\quad
W(I)-W(J)\le t.}
\tag{5.6}
\]
\begin{lemma}[Width interpolation]\label{lem:width-interpolation}
Let $\alpha_t=\width(P_t)$. Then
\[
\alpha_{t+1}\le\alpha_t\le\alpha_{t+1}+1
\qquad(0\le t<H).
\tag{5.7}
\]
Moreover,
\[
\alpha_0=b_{r+1},
\qquad
\alpha_H=b_r.
\tag{5.8}
\]
Consequently every integer between $b_r$ and $b_{r+1}$ occurs among the values $\alpha_t$.
\end{lemma}
\begin{proof}
For an integer $s$, let
\[
I_s=\{I\subseteq[r]:W(I)\le s\},
\qquad
F_s=\{I\subseteq[r]:W(I)>s\},
\]
with the induced Boolean order. We claim that
\[
\alpha_t
=
\max_{s\in\mathbb Z}
\bigl(\width(I_s)+\width(F_{s+t})\bigr).
\tag{5.9}
\]
Indeed, split an antichain of $P_t$ into its upper and lower parts. If the upper part is nonempty, choose an upper element $x_J$ of largest weight and put $s=W(J)$. Every upper element then belongs to $I_s$. If a lower element $x_I$ had $W(I)\le s+t$, then (5.6) would give $x_I\mid(q_t+1)x_J$, contradicting the antichain property. Hence every lower element belongs to $F_{s+t}$. The two parts therefore have sizes at most $\width(I_s)$ and $\width(F_{s+t})$. If the upper part is empty, take $s<-t$, so that $I_s$ is empty and $F_{s+t}$ is the whole Boolean lattice. If the lower part is empty, take $s\ge H$, so that $I_s$ is the whole Boolean lattice and $F_{s+t}$ is empty. Conversely, choose maximal antichains $A\subseteq I_s$ and $B\subseteq F_{s+t}$, place $A$ in the upper copy and $B$ in the lower copy, and consider $I\in B$, $J\in A$. Then $W(I)>s+t\ge W(J)+t$, so (5.6) excludes every lower-to-upper comparability. Thus $A\cup B$ is an antichain. This proves (5.9).

The weights (5.5) are pairwise distinct integers. Therefore, for fixed $s$, passing from $F_{s+t}$ to $F_{s+t+1}$ deletes at most one element. Deleting one element from a finite poset can lower its width by at most one. Formula (5.9) then gives
\[
\alpha_t-1\le\alpha_{t+1}\le\alpha_t,
\]
which is (5.7).

For $t=H$, condition (5.6) holds for every pair $I,J$, so every lower element lies below every upper element. Thus $P_H$ is the ordinal sum of two copies of $B_r$ and
\[
\alpha_H=b_r.
\]

For $t=0$, identify the lower copy with the layer of $B_{r+1}$ having last coordinate $0$, and the upper copy with the layer having last coordinate $1$. Comparabilities within each layer are the Boolean ones. A cross-comparability in $B_{r+1}$ has the form $(I,0)\le(J,1)$ with $I\subseteq J$; then $W(I)\le W(J)$, so (5.6) gives the same comparability in $P_0$. Thus every comparability of $B_{r+1}$ occurs in $P_0$, and hence
\[
\alpha_0\le b_{r+1}.
\]
Put
\[
h=\left\lfloor\frac{r+1}{2}\right\rfloor.
\]
Every $h$-element subset has weight at least $hA$, whereas every $(h-1)$-element subset has weight at most
\[
(h-1)A+(A-1)=hA-1.
\]
Therefore the $h$-th rank of the lower copy together with the $(h-1)$-st rank of the upper copy is an antichain in $P_0$. Its size is
\[
\binom rh+\binom r{h-1}
=\binom{r+1}{h}
=b_{r+1}.
\]
Thus $\alpha_0=b_{r+1}$. Since the integer sequence $\alpha_t$ decreases from $b_{r+1}$ to $b_r$ in steps of size at most one, it assumes every intermediate integer value.
\end{proof}

\begin{theorem}[Complete squarefree spectrum]\label{thm:squarefree-spectrum}
For every $r\ge1$,
\[
\boxed{
\mathcal W_r
=
\left\{1,2,\ldots,
\binom r{\lfloor r/2\rfloor}
\right\}.
}
\tag{5.10}
\]
Moreover, every value in this interval is attained by infinitely many squarefree integers with exactly $r$ prime factors.
\end{theorem}
\begin{proof}
We argue by induction on $r$. The assertion is immediate for $r=1$. Assume it for $r$. Corollary~\ref{cor:spectrum-preserve} preserves every value $1,\ldots,b_r$ when one new prime factor is added, with infinitely many realizations. Lemma~\ref{lem:width-interpolation} and the construction preceding it realize every value
\[
b_r,b_r+1,\ldots,b_{r+1}
\]
with $r+1$ prime factors, again with infinitely many realizations because the final prime $q_t$ can be chosen in infinitely many ways. Thus every value from $1$ to $b_{r+1}$ occurs. The upper bound in Theorem~\ref{thm:sq-upper} shows that no larger value is possible.
\end{proof}

\clearpage
\section{Optimal representations of divisibility posets}
By Theorem~\ref{thm:fixed-pattern}, repeated exponent blocks lead to widths of powers \(P^r\). Engel's limit theory expresses their asymptotics through the minimum representation variance \(V(P)\); we therefore determine \(V(D_m)\).

For a finite poset $P$ with uniform weight, a \emph{representation} is a function
\[
x:P\to\mathbb R
\]
such that
\[
x(q)-x(p)\ge1
\qquad(p<q).
\]
Let $V(P)$ be the minimum of the uniform variance $\Var(x)$ over all representations $x$, and put
\[
V_m=V(D_m).
\]
Engel's asymptotic product theorem states that for every finite poset $P$ which is not an antichain,
\[
\width(P^r)
\sim
\frac{|P|^r}{\sqrt{2\pi r V(P)}}
\qquad(r\to\infty)
\]
under uniform weights \cite{Engel1986,EngelBook}; see in particular \cite[Theorem~7.2.1]{EngelBook}. Our task is therefore to determine $V(D_m)$.

\subsection{A Ferrers lemma}
For graded posets, filter-average conditions are closely related to rank compression; see Bezrukov and Engel \cite{BezrukovEngel}. Since our Ferrers truncations need not be graded, we prove directly the filter-average inequality needed for optimality.

For integers
\[
w_0>w_1>\cdots>w_{H-1}>0,
\qquad w_H=0,
\]
let
\[
P(w)=\{(a,b):0\le b<H,\ 0\le a<w_b\}
\]
with the coordinatewise order, and define the natural height
\[
\rho(a,b)=a+b.
\]

\begin{lemma}[Ferrers lemma]\label{lem:ferrers}
Assume
\[
w_b-w_{b+1}\in\{1,2\}
\qquad(0\le b<H).
\]
Then $\rho$ is an optimal representation of $P(w)$.
\end{lemma}

It suffices to show that every filter has average height at least the global mean; then every monotone perturbation of \(\rho\) has nonnegative covariance with \(\rho\).

\begin{proof}
Let $\mu$ be the average of $\rho$ on $P(w)$. We first prove that every order ideal has average height at most $\mu$.

An ideal is described by row lengths
\[
w_b\ge x_b\ge x_{b+1}\ge0,
\qquad x_H:=0.
\]
Relax $x_0,\ldots,x_{H-1}$ to real variables and let $K$ be the resulting compact polytope. Define
\[
\Phi(x)=\sum_{b=0}^{H-1}
\left(
 bx_b+\frac{x_b(x_b-1)}2-\mu x_b
\right).
\]
For integral $x$, this is the total value of $\rho-\mu$ on the corresponding ideal. The function $\Phi$ is convex, so it is enough to prove $\Phi\le0$ at the vertices of $K$.

At a vertex, every maximal constant block
\[
x_s=\cdots=x_t=c
\]
is pinned either at $c=0$ or at $c=w_t$. Indeed, if $0<c<w_t$, maximality leaves positive slack in the neighboring monotonicity constraints (with the evident boundary modifications), and the upper bounds are also slack because $w_b\ge w_t>c$ on the block. For sufficiently small $\varepsilon>0$, shifting the whole block by $\pm\varepsilon$ gives two feasible points with midpoint $x$, contradicting extremality.

A positive constant block $x_s=\cdots=x_t=w_t$ contributes to the complement precisely the cells in rows $s,\ldots,t-1$ lying to the right of column $w_t-1$; a final zero block contributes the remaining full rows. Thus the complement of a vertex ideal is a disjoint union of wedges
\[
W_{s,t}
=
\{(a,b):s\le b<t,\ w_t\le a<w_b\},
\qquad 0\le s<t\le H.
\]
We show that every such wedge has average height at least $\mu$.

Fix $s$. Inside $W_{s,t}$, the average height of row $b$ is
\[
b+\frac{w_t+w_b-1}{2},
\qquad s\le b<t.
\]
These row averages are nondecreasing in $b$, because $w_b-w_{b+1}\le2$. For $t<H$,
\[
W_{s,t+1}=W_{s,t}\,\dot\cup\,R_{s,t},
\]
where
\[
R_{s,t}
=
\{(a,b):w_{t+1}\le a<w_t,\ s\le b\le t\}.
\]
The smallest row average in $W_{s,t}$ exceeds the average height of $R_{s,t}$ by
\[
\frac{s+w_s-w_{t+1}-t}{2}
\ge\frac12,
\]
because every row-length drop is at least $1$, so $w_s-w_{t+1}\ge t-s+1$. Hence adjoining $R_{s,t}$ can only decrease the average, and therefore
\[
\operatorname{av}(W_{s,t})
\ge
\operatorname{av}(W_{s,H}).
\]
Now $W_{s,H}$ is simply the union of the full rows $s,\ldots,H-1$. The average height of the full $b$-th row is
\[
b+\frac{w_b-1}{2},
\]
and these row averages are again nondecreasing. Removing the initial rows can therefore only increase the overall average, so
\[
\operatorname{av}(W_{s,H})\ge\mu.
\]
Thus every wedge has average height at least $\mu$. The complement of every vertex ideal has average at least $\mu$, hence every vertex ideal has average at most $\mu$. Since $\Phi$ is convex, the same holds for every ideal.

Equivalently, every nonempty filter has average $\rho$ at least $\mu$. Let $x$ be any representation and write
\[
x=\rho+y.
\]
The set $P(w)$ is an order ideal in $\N_0^2$, so every cover changes exactly one coordinate by $1$. Hence $\rho$ increases by exactly $1$ along every cover relation, and $y$ is nondecreasing. Every nondecreasing function on a finite poset is a constant plus a nonnegative linear combination of indicator functions of filters: list its distinct values $c_0<\cdots<c_m$ and write
\[
y=c_0+\sum_{j=1}^m(c_j-c_{j-1})\,\mathbf 1_{\{y\ge c_j\}},
\]
where each upper level set $\{y\ge c_j\}$ is a filter. It follows from the filter inequality that
\[
\Cov(\rho,y)\ge0.
\]
Therefore
\[
\Var(x)
=
\Var(\rho)+2\Cov(\rho,y)+\Var(y)
\ge
\Var(\rho).
\]
Thus $\rho$ is optimal.
\end{proof}

\subsection{The logarithmic Ferrers diagrams}
For $T\ge1$, define
\[
Q_T=\{(a,b)\in\N_0^2:2^a3^b\le T\}.
\]
Put
\[
N(T)=|Q_T|,
\qquad
\mu(T)=\frac1{N(T)}\sum_{(a,b)\in Q_T}(a+b),
\]
and
\[
v(T)=\frac1{N(T)}
\sum_{(a,b)\in Q_T}(a+b-\mu(T))^2.
\]

Figure~\ref{fig:q12} shows the shape to keep in mind.

\begin{figure}[ht]
\centering
\begin{tikzpicture}[x=1.35cm,y=1.05cm,
    pt/.style={circle,draw,minimum size=7mm,inner sep=0pt,font=\small}]
  \node[pt] at (0,0) {1};
  \node[pt] at (1,0) {2};
  \node[pt] at (2,0) {4};
  \node[pt] at (3,0) {8};
  \node[pt] at (0,1) {3};
  \node[pt] at (1,1) {6};
  \node[pt] at (2,1) {12};
  \node[pt] at (0,2) {9};

  \draw[->] (-0.55,-0.55) -- (3.65,-0.55) node[right] {\(a\)};
  \draw[->] (-0.55,-0.55) -- (-0.55,2.55) node[above] {\(b\)};

  \draw (0.35,0)--(0.65,0);
  \draw (1.35,0)--(1.65,0);
  \draw (2.35,0)--(2.65,0);
  \draw (0,0.35)--(0,0.65);
  \draw (1,0.35)--(1,0.65);
  \draw (2,0.35)--(2,0.65);
  \draw (0,1.35)--(0,1.65);
\end{tikzpicture}
\caption{The diagram \(Q_{12}\). The label at \((a,b)\) is \(2^a3^b\). Moving right multiplies by \(2\), moving up multiplies by \(3\), and the natural height is \(a+b\). The row lengths are \(4,3,1\).}
\label{fig:q12}
\end{figure}
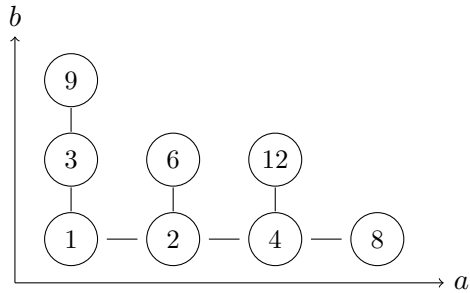

The nonempty rows of $Q_T$ have lengths
\[
w_b=1+\left\lfloor\log_2\frac{T}{3^b}\right\rfloor.
\]
Since $1<\log_2 3<2$, successive nonempty row lengths differ by $1$ or $2$; the last nonempty row has length $1$ or $2$. Lemma~\ref{lem:ferrers} therefore applies to every $Q_T$.

We need one quantitative consequence.

\begin{lemma}[Scaling lemma]\label{lem:mu-scale}
For every integer $T\ge1$,
\[
\boxed{
\mu(2T)-\mu(T)\ge\frac12.
}
\]
Consequently,
\[
\boxed{
\mu(4T)-\mu(T)\ge1.
}
\]
\end{lemma}
\begin{proof}
Let $h$ be the number of points of the first column of $Q_{2T}$, equivalently
\[
h=\#\{b\ge0:3^b\le2T\}.
\]
For $0\le b<h$, let $w_b$ be the number of $a\ge0$ such that $2^a3^b\le T$; the last $w_b$ is allowed to be $0$. Put
\[
N=N(T)=\sum_{b=0}^{h-1}w_b
\]
and
\[
c_b=w_b+2b.
\]
The sequence $(w_b)$ is nonincreasing, while
\[
c_{b+1}-c_b
=2-(w_b-w_{b+1})\in\{0,1\}
\]
whenever both sides are defined. Hence $(c_b)$ is nondecreasing. Since $(w_b)$ and $(c_b)$ are oppositely ordered,
\[
h\sum_b w_b c_b-N\sum_b c_b
=
\sum_{0\le i<j<h}(w_i-w_j)(c_i-c_j)
\le0.
\]
Therefore
\[
\frac{\sum_b w_b c_b}{N}
\le
\frac1h\sum_b c_b.
\]
Now
\[
2\mu(T)+1
=
\frac{\sum_b w_b c_b}{N},
\]
whereas
\[
\sum_b c_b
=N+h(h-1).
\]
Therefore
\[
\mu(T)
\le
\frac12\left(\frac Nh+h-2\right).
\tag{6.1}
\]

The set $Q_{2T}$ is the disjoint union of the translate
\[
(a,b)\mapsto(a+1,b)
\]
of $Q_T$ and the first column
\[
\{(0,b):0\le b<h\}.
\]
Thus
\[
N(2T)=N+h
\]
and
\[
\mu(2T)-\mu(T)
=
\frac{N+h\bigl((h-1)/2-\mu(T)\bigr)}{N+h}.
\]
By (6.1),
\[
\frac{h-1}{2}-\mu(T)
\ge
\frac{h-N}{2h}.
\]
The numerator above is therefore at least
\[
N+\frac{h-N}{2}
=\frac{N+h}{2},
\]
which proves
\[
\mu(2T)-\mu(T)\ge\frac12.
\]
Applying this twice gives the second assertion.
\end{proof}

\begin{lemma}\label{lem:mu-monotone}
The function $T\mapsto\mu(T)$ is nondecreasing on the positive integers.
\end{lemma}
\begin{proof}
If $S\le T$, then $Q_S$ is an order ideal of $Q_T$. By Lemma~\ref{lem:ferrers}, every ideal of $Q_T$ has average height at most the average height of $Q_T$. Hence $\mu(S)\le\mu(T)$.
\end{proof}

\subsection{Critical dilation}
The scaling lemma already gives the precise dilation needed for the prime-core argument.

\begin{corollary}[Critical dilation]\label{cor:critical-dilation}
The least positive integer $L$ such that
\[
\mu(LT)-\mu(T)\ge1
\qquad(T\ge1)
\]
is
\[
\boxed{L=4.}
\]
\end{corollary}
\begin{proof}
Lemma~\ref{lem:mu-scale}, applied twice, gives
\[
\mu(4T)-\mu(T)\ge1
\]
for every $T$. To see that no smaller integer works, take $T=1$. For $L=1,2,3$, the sets $Q_L$ have mean heights $0$, $1/2$, and $2/3$, respectively, all strictly below $1=\mu(1)+1$.
\end{proof}

Thus the unit-gain scale for the $2,3$-smooth truncations is exactly $4$. Since the next prime is $5$, this is the numerical separation that makes all prime directions at least $5$ redundant in the optimal representation.

\subsection{Prime-core reduction}\label{subsec:why23}
For a finite set \(\mathcal P\) of primes, let \(D_m^{(\mathcal P)}\) be the poset on \(\{1,\ldots,m\}\) in which \(a\preceq b\) when \(a\mid b\) and every prime divisor of \(b/a\) belongs to \(\mathcal P\). Also put
\[
S_{\mathcal P}(T)=\{n\le T:\text{ every prime divisor of }n\text{ belongs to }\mathcal P\},
\]
ordered by divisibility, let \(\rho(n)=\Omega(n)\), and write \(\mu_{\mathcal P}(T)\) for the average of \(\rho\) on \(S_{\mathcal P}(T)\).

The mechanism behind the \(2,3\)-reduction is the following general principle.

\begin{theorem}[Prime-core reduction principle]\label{thm:core-reduction}
Suppose that
\begin{enumerate}[label=\textup{(\roman*)}]
\item \(\rho\) is an optimal representation of \(S_{\mathcal P}(T)\) for every \(T\ge1\);
\item \(\mu_{\mathcal P}(T)\) is nondecreasing;
\item for some integer \(L\ge2\),
\[
\mu_{\mathcal P}(LT)-\mu_{\mathcal P}(T)\ge1
\qquad(T\ge1);
\]
\item every prime outside \(\mathcal P\) is at least \(L\).
\end{enumerate}
Then
\[
V(D_m)=V(D_m^{(\mathcal P)})
\qquad(m\ge1).
\]
More precisely, write uniquely \(n=us\), where \(s\) is \(\mathcal P\)-smooth and \(u\) has no prime divisor in \(\mathcal P\). Then
\[
x_m(n)=\Omega(s)-\mu_{\mathcal P}\!\left(\left\lfloor\frac m u\right\rfloor\right)
\]
is an optimal representation of \(D_m\).
\end{theorem}
\begin{proof}
For fixed \(u\), the corresponding component of \(D_m^{(\mathcal P)}\) is isomorphic to
\[
S_{\mathcal P}\!\left(\left\lfloor\frac m u\right\rfloor\right).
\]
By (i), \(\rho\) is optimal on each component. Since different components are disconnected, they may be translated independently. Centering each component at mean \(0\) therefore gives the displayed optimal representation of \(D_m^{(\mathcal P)}\).

It remains to check the missing divisibility relations. Suppose \(n<n'\) and \(n\mid n'\). Write
\[
n=us,
\qquad
n'=uvsw,
\]
where \(s,w\) are \(\mathcal P\)-smooth and \(u,v\) have no prime divisor in \(\mathcal P\). Put
\[
T=\left\lfloor\frac m u\right\rfloor,
\qquad
T'=\left\lfloor\frac m{uv}\right\rfloor.
\]
Then
\[
x_m(n')-x_m(n)
=
\Omega(w)+\mu_{\mathcal P}(T)-\mu_{\mathcal P}(T').
\]
If \(w>1\), then \(\Omega(w)\ge1\), and (ii) gives a difference at least \(1\).

If \(w=1\), then \(v>1\). By (iv), \(v\ge L\), so \(LT'\le T\). Hence (ii) and (iii) give
\[
\mu_{\mathcal P}(T)-\mu_{\mathcal P}(T')
\ge
\mu_{\mathcal P}(LT')-\mu_{\mathcal P}(T')
\ge1.
\]
Thus the centered representation is feasible for the full poset \(D_m\). Since adding constraints can only increase the minimum variance, it is optimal for both posets.
\end{proof}

We now apply the principle to \(\mathcal P=\{2,3\}\). In this case \(S_{\mathcal P}(T)\) is identified with \(Q_T\), so optimality follows from Lemma~\ref{lem:ferrers}, monotonicity from Lemma~\ref{lem:mu-monotone}, and the critical dilation is \(4\) by Corollary~\ref{cor:critical-dilation}. The first omitted prime is \(5\).

\begin{theorem}[$2,3$-reduction]\label{thm:23}
For every \(m\ge1\),
\[
\boxed{V(D_m)=V(D_m^{(2,3)}).}
\]
More precisely, if
\[
n=u2^a3^b,
\qquad (u,6)=1,
\]
then
\[
\boxed{
x_m(n)=a+b-\mu\!\left(\left\lfloor\frac m u\right\rfloor\right)
}
\tag{6.2}
\]
is an optimal representation of \(D_m\).
\end{theorem}
\begin{proof}
Apply Theorem~\ref{thm:core-reduction} with \(\mathcal P=\{2,3\}\) and \(L=4\).
\end{proof}

The arithmetic reason is therefore concentrated in
\[
\boxed{3<4<5.}
\]
The prime \(3\) lies below the unit-gain scale \(4=2^2\), while the next prime already lies beyond it. Thus the \(2\)- and \(3\)-directions must be retained, but every prime direction at least \(5\) is absorbed by the centering distance between components.

The pair \(\{2,3\}\) is sharp in two complementary senses.

\begin{proposition}[Minimal universal prime core]\label{prop:minimal-prime-core}
For a set \(\mathcal P\) of primes,
\[
\boxed{
V(D_m^{(\mathcal P)})=V(D_m)\text{ for every }m
\quad\Longleftrightarrow\quad
\{2,3\}\subseteq\mathcal P.
}
\tag{6.3}
\]
\end{proposition}
\begin{proof}
If \(\{2,3\}\subseteq\mathcal P\), then
\[
D_m^{(2,3)}\subseteq D_m^{(\mathcal P)}\subseteq D_m
\]
in the sense of inclusion of order relations. Adding constraints can only increase the minimum representation variance, so Theorem~\ref{thm:23} gives equality throughout.

Conversely, if \(2\notin\mathcal P\), then \(D_2^{(\mathcal P)}\) is an antichain and has optimal variance \(0\), whereas \(D_2\) is a two-element chain and has optimal variance \(1/4\). Suppose next that \(2\in\mathcal P\) but \(3\notin\mathcal P\). Then \(D_3^{(\mathcal P)}\) consists of the chain \(1<2\) and the isolated point \(3\), so its optimal variance is \(1/6\). For the full poset \(D_3\), the centered optimal representation has values \(-2/3,1/3,1/3\), and hence
\[
V(D_3)=\frac29>\frac16.
\]
Thus both \(2\) and \(3\) are necessary.
\end{proof}

The proposition shows that \(\{2,3\}\) is the smallest universal set of prime directions that determines the exact variance of every \(D_m\).

\subsection{Exact formula and limiting constant}
Theorem~\ref{thm:23} immediately gives an exact finite formula.

\begin{corollary}\label{cor:Vm}
For every $m\ge1$,
\[
\boxed{
V_m
=
\frac1m
\sum_{\substack{u\le m\\(u,6)=1}}
N\!\left(\left\lfloor\frac m u\right\rfloor\right)
 v\!\left(\left\lfloor\frac m u\right\rfloor\right).
}
\tag{6.4}
\]
\end{corollary}
\begin{proof}
Each component indexed by $u$ has $N(\lfloor m/u\rfloor)$ elements and, after centering, variance $v(\lfloor m/u\rfloor)$. All component means are $0$, so the variance of the union is the size-weighted average of the component variances.
\end{proof}

\begin{theorem}\label{thm:C}
The sequence \(V_m\) converges. More precisely,
\[
\boxed{
V_m=C+O\!\left(\frac{(\log(2m))^4}{m}\right),
}
\]
where
\[
\boxed{
C=\frac13\sum_{t\ge1}\frac{N(t)v(t)}{t(t+1)}.
}
\tag{6.5}
\]
Numerically,
\[
C=1.0367934201598\ldots.
\]
\end{theorem}
\begin{proof}
Put
\[
F(t)=N(t)v(t).
\]
Since both coordinates in \(Q_t\) are \(O(\log(2t))\),
\[
F(t)=O((\log(2t))^4),
\]
so the series in (7.10) converges absolutely.

We first note that \(F(t)\) is nondecreasing. When a new point of height \(h\) is added to a set of \(N\) heights with mean \(\mu\) and unnormalized variance \(F\), the new unnormalized variance is
\[
F+\frac{N}{N+1}(h-\mu)^2.
\]
Thus enlarging \(Q_t\) cannot decrease \(F(t)\).

Let
\[
B(x)=\#\{u\le x:(u,6)=1\}.
\]
Periodicity modulo \(6\) gives
\[
B(x)=\frac{x}{3}+E(x),
\]
where \(E\) is bounded on \([0,\infty)\). Grouping Corollary~\ref{cor:Vm} according to
\[
t=\left\lfloor\frac m u\right\rfloor
\]
gives
\[
V_m=\frac1m\sum_{t=1}^m F(t)A_m(t),
\]
where
\[
A_m(t)=B(m/t)-B(m/(t+1)).
\]
Hence
\[
A_m(t)
=
\frac{m}{3t(t+1)}
+E(m/t)-E(m/(t+1)).
\]
Write \(e_t=E(m/t)\). The discrepancy from the first \(m\) terms of (7.10) is therefore
\[
\frac1m\sum_{t=1}^m F(t)(e_t-e_{t+1}).
\]
Summation by parts yields
\[
\sum_{t=1}^m F(t)(e_t-e_{t+1})
=
F(1)e_1
+\sum_{t=2}^m\bigl(F(t)-F(t-1)\bigr)e_t
-F(m)e_{m+1}.
\]
Since \(E\) is bounded and \(F\) is nondecreasing, the right-hand side is \(O(F(m))\). Thus the discrepancy is
\[
O\!\left(\frac{(\log(2m))^4}{m}\right).
\]
Finally,
\[
\sum_{t>m}\frac{F(t)}{t(t+1)}
=
O\!\left(\frac{(\log(2m))^4}{m}\right)
\]
by comparison with the corresponding integral. Combining the two estimates proves the theorem.
\end{proof}

\begin{remark}\label{rem:prob}
The constant $C$ has a probabilistic interpretation. Let $\mathcal H=\{2^a3^b:a,b\ge0\}$, choose $N\in\mathcal H$ with
\[
\Pr(N=n)=\frac1{3n},
\]
let $U$ be independent and uniform on $(0,1)$, and put $T=\lfloor N/U\rfloor$. For $n\le t$,
\[
\Pr(N=n,T=t)=\frac1{3t(t+1)},
\]
so conditional on $T=t$, the variable $N$ is uniform on the $3$-smooth integers at most $t$. Hence
\[
C=\mathbb E\!\left[\Var(\Omega(N)\mid T)\right].
\]
Writing $N=2^A3^B$ shows that $A$ and $B$ are independent geometric variables with $\Var(A+B)=11/4$. The law of total variance therefore gives
\[
\boxed{\frac{11}{4}=C+\Var(\mu(T)).}
\]
Thus $C$ measures the residual variance of $\Omega(N)$ after the truncation parameter $T$ is known.
\end{remark}

\subsection{Repeated exponent patterns}
The preceding variance calculation applies to more than equal exponents. Fix an exponent vector
\[
\mathbf e=(e_1,\ldots,e_s),
\qquad e_j\ge1,
\]
and let \(M_{\mathbf e}(r)\) be the largest value of \(a(n)\) among integers whose prime-exponent multiset consists of \(r\) copies of \(\mathbf e\). Thus
\[
n=\prod_{\nu=1}^r\prod_{j=1}^s p_{\nu,j}^{e_j},
\]
with all primes \(p_{\nu,j}\) distinct. Put
\[
P_{\mathbf e}=D_{e_1+1}\times\cdots\times D_{e_s+1}.
\]
By Theorem~\ref{thm:fixed-pattern},
\[
M_{\mathbf e}(r)=\width(P_{\mathbf e}^r).
\]
Thus a repeated exponent pattern is the ordinary power of one fixed finite poset, so Engel's asymptotic product theorem applies directly. His product theorem for optimal representations gives
\[
V(P\times Q)=V(P)+V(Q)
\]
for finite uniformly weighted posets \cite[Theorem~4.6.6]{EngelBook}. Hence
\[
V(P_{\mathbf e})=\sum_{j=1}^sV_{e_j+1}.
\]
We obtain the following.

\begin{corollary}\label{cor:mixed-exp-asymp}
For every fixed exponent vector \(\mathbf e=(e_1,\ldots,e_s)\),
\[
\boxed{
M_{\mathbf e}(r)
\sim
\frac{\left(\prod_{j=1}^s(e_j+1)\right)^r}
{\sqrt{2\pi r\sum_{j=1}^sV_{e_j+1}}}
\qquad(r\to\infty).
}
\]
\end{corollary}
\begin{proof}
By Theorem~\ref{thm:fixed-pattern}, $M_{\mathbf e}(r)=\width(P_{\mathbf e}^r)$. Engel's asymptotic product theorem therefore gives
\[
M_{\mathbf e}(r)
\sim
\frac{|P_{\mathbf e}|^r}{\sqrt{2\pi rV(P_{\mathbf e})}}.
\]
Since $|P_{\mathbf e}|=\prod_{j=1}^s(e_j+1)$ and $V(P_{\mathbf e})=\sum_{j=1}^sV_{e_j+1}$, the stated formula follows.
\end{proof}

\section{Minimum widths and open problems}
The maximal problem is governed by eliminating cross-divisibilities; the opposite extremal question asks how strongly they can collapse the universal product poset. For an exponent vector $\mathbf e=(e_1,\ldots,e_r)$, put
\[
L(\mathbf e)=
\min_{p_1,\ldots,p_r}
 a(p_1^{e_1}\cdots p_r^{e_r}),
\]
where the minimum is over distinct primes.

For a single exponent, Theorem~\ref{thm:primepower} gives
\[
L(e)=\left\lceil\frac{e+1}{2}\right\rceil.
\]
A simple construction gives
\[
L(e,2)\le2L(e).
\]
Indeed, for $m=p^e$ let $M$ be the least common multiple of the elements of $S(m)$ and choose a prime $q\equiv-1\pmod M$. Then
\[
S(mq^2)=S(m)\cup(q+1)S(m)\cup(q^2+q+1)S(m),
\]
and a minimum chain cover of $S(m)$ covers the first two copies consecutively, while a second copy of the same cover handles the third. Exact computations in the first cases suggest that this upper bound is sharp.

\begin{conjecture}\label{conj:L2}
For every $e\ge1$,
\[
\boxed{
L(e,2)
=2\left\lceil\frac{e+1}{2}\right\rceil.
}
\]
\end{conjecture}

More generally, determining $L(e_1,\ldots,e_r)$ asks how far arithmetic cross-divisibilities can collapse the universal index poset. This completes the contrast running through the paper: in the maximal problem cross-divisibilities are eliminated, in the squarefree spectrum they are tuned through a threshold, and in the variance problem only the prime directions $2$ and $3$ remain active.

\section*{Use of artificial intelligence}

\noindent OpenAI ChatGPT was used to suggest possible proof approaches, stress-test selected arguments, and assist with organization and wording. The author independently verified the mathematical claims, computations, and references and takes full responsibility for the final content.

\end{document}